\documentclass{llncs}
\usepackage{amssymb}
\usepackage{graphicx,color}
\newcommand{\ignore}[1]{}
\newcommand{\com}[1]{}
\newcommand{\boxi}{\ensuremath{\mathrm{box}}}
\newcommand{\bbox}{\rule{0.6em}{0.6em}}

\begin{document}

\title{Chordal Bipartite Graphs with High Boxicity}
\author{L. Sunil Chandran\and Mathew C. Francis\and Rogers Mathew}
\institute{Department of Computer Science and Automation, Indian Institute of
Science, Bangalore -- 560012, India.\\
\texttt{\{sunil,mathew,rogers\}@csa.iisc.ernet.in}}
\maketitle
\bibliographystyle{plain}
\begin{abstract}
The boxicity of a graph $G$ is defined as the minimum
integer $k$ such that $G$ is an intersection graph of axis-parallel
$k$-dimensional boxes. Chordal bipartite graphs are bipartite graphs that do
not contain an induced cycle of length greater than 4. It was conjectured
by Otachi, Okamoto and Yamazaki that chordal bipartite graphs have
boxicity at most 2. We disprove this conjecture by exhibiting an infinite family
of chordal bipartite graphs that have unbounded boxicity.

\end{abstract}
\section{Introduction}
A graph $G$ is an \emph{intersection graph} of sets from a family of sets
$\mathcal{F}$, if there exists $f:V(G)\rightarrow \mathcal{F}$ such that
$(u,v)\in E(G)\Leftrightarrow f(u)\cap f(v)\not=\emptyset$.
\com{ More formally, 
let $G=(V,E)$ be a graph and let $\mathcal{F} = \{S_v~|~v \in V(G)\}$ be a
family of sets. Then, $G$ is an 
intersection graph if there exists a mapping $f:V(G)\rightarrow \mathcal{F}$
such that 
for any two vertices $u,v \in V(G)$, $(u,v) \in E(G)$ if and only if $f(u)
\cap f(v) \neq \emptyset$.}
An \emph{interval graph} is an intersection graph
in which the set assigned to each vertex is a closed interval on the real line. 
In other words, interval graphs are intersection graphs of closed intervals on
the real line.
An \emph{axis-parallel $k$-dimensional box} in $\mathbb{R}^k$
\com{a $k$-dimensional space} is the
Cartesian product $R_1 \times R_2 \times \cdots \times R_k$, where each $R_i$ is
an interval of the form $[a_i, b_i]$ on the real line.
\com{ Thus, any point
$p=(x_1,x_2, \cdots ,x_k)$ is on or inside the $k$-dimensional box if 
$(x_1,x_2, \cdots ,x_k) \in R_1 \times R_2 \times \cdots \times R_k$.}
\emph{Boxicity} of any graph $G$ (denoted by $\boxi(G)$) is the minimum integer
$k$ such that $G$ is an intersection graph of axis-parallel $k$-dimensional
boxes in $\mathbb{R}^k$. Note that interval graphs are exactly those
graphs with boxicity at most 1.

The concept of boxicity was introduced by F. S. Roberts in the year 1969
\cite{Roberts}. It finds applications in niche overlap (competition) in
ecology and to problems of fleet maintenance in operations research
(see \cite{CozRob}). Roberts proved that the boxicity of any graph 
on $n$ vertices is upper bounded by $\lfloor \frac{n}{2}\rfloor$. He also 
showed that a complete $\frac{n}{2}$-partite graph with 2 vertices in each part 
has its boxicity equal to $\frac{n}{2}$. Various other upper bounds on
boxicity in terms of graph parameters such as maximum degree and treewidth
were proved by Chandran, Francis and Sivadasan. In \cite{tech-rep}  they showed
that, for any graph $G$ on $n$ vertices having maximum degree $\Delta$,
$\boxi(G) \leq (\Delta + 2)\ln n $.
They also upper bounded boxicity solely in terms of the maximum degree $\Delta$
of a graph by showing that $\boxi(G) \leq 2\Delta^2$ \cite{CFNMaxdeg}.
This means that the boxicity of degree bounded graphs do not is bounded no
matter what the size of the vertex set is.
\com{vary with the size of their vertex set.}
It was shown in \cite{CN05} by Chandran and Sivadasan that $\boxi(G) \leq
\mathrm{tw}(G)+ 2$, where $\mathrm{tw}(G)$ denotes the treewidth of graph $G$. 

Cozzens \cite{Coz} proved that given a graph, the problem of computing its 
boxicity is NP-hard. Several attempts have been made to find good upper bounds
\com{compute} for the 
boxicity of special classes of graphs. It was shown by Thomassen in
\cite{Thoma1} that planar graphs have boxicity at most 3. Meanwhile,
Scheinerman \cite{Scheiner} proved that outerplanar graphs have boxicity at most
2. The boxicity of split graphs was investigated by
Cozzens and Roberts \cite{CozRob}. Apart from these results, not much is 
known about the boxicity of most of the well-known graph classes.
\subsection{Chordal Bipartite Graphs (CBGs)}
A bipartite graph $G$ is a \emph{chordal bipartite graph} (CBG) if $G$ does not have
an induced cycle of length greater than 4. In other words, all induced 
cycles in such a bipartite graph will be of length exactly equal to 4. 
Chordal bipartite graphs were introduced by Golumbic and Goss \cite{GoluGoss},
as a natural bipartite analogue of chordal graphs.
Chordal bipartite graphs are a well studied class of graphs 
and several characterizations have been found, such as by 
the elimination scheme, minimal edge separators, $\Gamma$-free matrices etc.
(refer \cite{Golu}). 
\subsection{Our Result}
\com{
It was shown in \cite{Bellantoni} that a bipartite graph has boxicity at most
2 if and only if it is a grid intersection graph.}
In 2007, Otachi, Okamoto and Yamazaki \cite{OOY} proved that 
$P_6$-free chordal bipartite graphs have boxicity at most $2$.
\com{ and are therefore grid intersection graphs.}
In the same paper, they also conjectured that the boxicity of any chordal
bipartite graph is upper bounded by the same constant 2.
\com{which would imply that all
chordal bipartite graphs are grid intersection graphs.}
We disprove this conjecture by showing that there exist chordal 
bipartite graphs with arbitrarily high boxicity. This result also implies
that the class of chordal bipartite
graphs is incomparable with the class of ``grid intersection graphs''
(see \cite{Bellantoni}).
\section{Definitions and Notations}
Let $V(G)$ and $E(G)$ denote the vertex set and edge set respectively of 
a graph $G$. For any $S \subseteq V(G)$, let $G-S$ denote the graph 
induced by the vertex set $V(G) \setminus S$ in $G$. In this paper, we
consider only simple, finite, undirected graphs. In a graph $G$, for any
$u \in V(G)$, $N(u)$ denotes its neigbourhood in $G$, i.e. $N(u) = \{v~|~(u,v)
\in E(G)\}$. Also, $N[u]$ denotes the closed neighbourhood of $u$ in $G$, i.e.
$N[u] = N(u) \cup \{u\}$. A graph $G$ is a bipartite graph if there is a
partition of $V(G)$ into two sets $A$ and $B$ such that both $A$ and $B$ induce
independent sets in $G$. We call $\{A,B\}$ the \emph{bipartition} of the
bipartite graph $G$.
Given a tree $T$ and two vertices $u$ and $v$ in $T$,
we denote by $uTv$ the unique path in $T$ between $u$ and $v$ (including $u$ and
$v$).
If $G,\,G_1,\,G_2,\ldots,\, G_k$ 
are $k+1$ graphs, where $V(G) = V(G_1) = V(G_2) = \cdots = V(G_k)$,
then we say that $G = G_1 \cap G_2 \cap \cdots \cap G_k$ if 
$E(G) = E(G_1) \cap E(G_2) \cap \cdots \cap E(G_k)$. 

\subsection{Interval Graphs and Boxicity}

\com{We know that interval graphs are intersection graphs in which the set 
assigned to each vertex is an interval on the real line. So any 
interval graph will have an interval representation of 
itself.}
Since an interval graph is the intersection graph of closed intervals on the
real line, for every interval graph $I$, there exists a function
$f:V(I)\rightarrow \{X\subseteq\mathbb{R}~|~X\mbox{ is a closed interval}\}$,
such that for $u,v\in V(I)$, $(u,v)\in E(I)\Leftrightarrow f(u)\cap
f(v)\not=\emptyset$.
The function $f$ is called an \emph{interval representation} of the interval
graph $I$. Note that it is possible for an interval graph to have more than 
one interval representation.
Given a closed interval $X=[y,z]$, we define $l(X):=y$ and $r(X):=z$. Also
note that by the definition of an interval, if $[y,z]$ is an interval, $y\leq
z$.
\com{ For any interval graph $I$, let 
$int(I)$ denote any one such representation. For any $u \in V(I)$, 
let $interval(u,I)$ denote the interval assigned to vertex $u$ 
in $int(I)$. Let $left(u,I)$ ($right(u,I)$) denote the leftmost 
(rightmost) endpoint of $interval(u,I)$. So $interval(u,I)$ is 
the interval $[left(u,I), right(u,I)]$ on the real line.
}
For any two 
intervals $[y_1, z_1], [y_2,z_2]$ on the real line, we say that $[y_1, z_1] <
[y_2,z_2]$ if $z_1 < y_2$. Clearly, $[y_1, z_1] \cap [y_2,z_2] = \emptyset$ if
and only if $[y_1, z_1] < [y_2,z_2]$ or $[y_2,z_2] < [y_1, z_1]$.
\com{
Hence for any 
$u,v \in V(I)$, $interval(u,I) < interval(v,I)$ 
if $left(u,I) \leq right(u,I) < left(v,I) \leq right(v,I)$.
Also $interval(u,I) \cap interval(v,I) = \emptyset$ if 
$interval(u,I) < interval(v,I)$ or $interval(v,I) < interval(u,I)$. 
}

Let $I$ be an interval graph and $f$ an interval representation of $I$.
\com{$k \in \mathbb{N}$ be a positive natural number and.} Let $y,z \in
\mathbb{R}$ with $y \leq z$. Then any set of vertices, say $S=\{u_1, u_2, \ldots
, u_k\}$ where $S \subseteq V(I)$ and $k>0$, is said to	
``overlap in the region $[y,z]$ in $f$'' if each $f(u_i)$ 
(where $1 \leq i \leq k$) contains the region $[y,z]$, i.e. for 
each $u_i \in S$, $l(f(u_i)) \leq y \leq z \leq r(f(u_i))$.
\vspace{0.01in}

A graph $G$ is \emph{chordal} if it does not contain any induced 
cycle of length greater than 3.
\com{ An \emph{asteroidal triple} in 
a graph $G$ is a set $S$ of 3 independent vertices  
such that for any $u \in S$, there exists a path between the other two 
vertices in $S$ which does not contain any vertex adjacent to
 $u$. A graph $G$ is said to be \emph{asteroidal triple-free}
 or AT-free if it does not contain any asteroidal triple. }
\com{The following lemma characterizes interval graphs.}
The following is a well known fact about interval graphs.
\begin{lemma}\com{[Lekkerkerker and Boland\cite{Lekker}]}\label{lekkerlem}
All interval graphs are chordal.
\end{lemma}

We have seen that interval graphs are intersection graphs of 
intervals on the real line. The following lemma gives the relationship between
intersection graphs of axis-parallel $k$-dimensional boxes and interval graphs.
\begin{lemma}[Roberts\cite{Roberts}]\label{robertslem}
For any graph $G$, $\boxi(G) \leq b$ if and only if
there exist $b$ interval graphs $I_1, I_2, \ldots, I_b$,
with $V(G)=V(I_1)=V(I_2)=\cdots=V(I_b)$ such that $G = I_1 
\cap I_2 \cap \cdots \cap I_b$.
\end{lemma}
From the above lemma, we can say that the boxicity of a 
graph $G$ is the minimum $b$ for which there exist $b$ interval graphs
$I_1,\ldots,I_b$ such that $G =I_1 \cap I_2 \cap \cdots \cap I_b$. Note that if
$G=I_1 \cap I_2 \cap \cdots \cap I_b$, then each $I_i$ is a supergraph
of $G$ and also for every pair of vertices $u,v\in V(G)$ such that $(u,v)\not\in
E(G)$, $(u,v) \notin E(I_i)$, for some $i$.

\ignore{
An edge $e=(u,v)$ of a bipartite graph $G$ is \emph{bisimplicial} 
if $N(u) \cup N(v)$ induces a complete bipartite subgraph of $G$. 
Let $|E(G)| = m$ and let $\sigma = [e_1,e_2, \ldots , e_m]$ be a 
sequence of all the edges of $G$. Let $G_0 = G$ and for any 
$1 \leq i \leq m$, let $G_i$ be the graph 
obtained after removing edge $e_i$ from $G_{i-1}$ i.e. $V(G_i) = 
V(G_{i-1})$ and $E(G_i) = E(G_{i-1})\setminus\{e_i\}$. So $V(G_1) = 
V(G_0)$ and $E(G_1) = E(G_0)\setminus\{e_1\}$. Also, note that $G_m$
will be a graph with no edges. We say that
 $\sigma$ is a  \emph{perfect edge without vertex elimination ordering} 
for $G$ if each edge $e_i$ is bisimplicial in the graph $G_{i-1}$ from 
which it is removed to get the graph $G_i$. The following lemma appears 
in \cite{BLS}.
\begin{lemma}\label{PEOlemma}
A graph $G$ is chordal bipartite if and only if it has a perfect 
edge without vertex elimination ordering. 
\end{lemma}
}
\subsection{Strongly Chordal Graphs and Chordal Bipartite Graphs}

A chordal graph is \emph{strongly chordal} if it 
does not contain any induced trampoline (refer \cite{Farber83}). 
Two vertices $u$ and $v$ in a graph are said to be \emph{compatible} 
if $N[u] \subseteq N[v]$ or vice versa. Otherwise they are said to be 
\emph{non-compatible}. A vertex $v$ in a graph $G$ is a \emph{simple vertex}
if for any $x,y \in N[v]$, $x$ and $y$ are compatible. An ordering $v_1,
\ldots, v_n$ of vertices of a graph $G$ is said to be a \emph{simple
elimination ordering} if
for each $i$, the vertex $v_i$ is a simple vertex in the graph induced by
the vertices $\{v_i,\ldots,v_n\}$ in $G$. The following characterization 
of strongly chordal graphs is from page 78 of \cite{BLS}. 
\begin{lemma}\com{[Brandstadt, Le, Spinrad \cite{BLS}]}\label{stronglemma}
A graph is strongly chordal if and only if it admits a simple elimination
ordering.
\end{lemma}

For a bipartite graph $G$ with bipartition $\{A,B\}$, we denote by $C_A(G)$
($C_B(G)$) the split graph obtained from $G$ by adding edges between every pair
of vertices in $A$ ($B$).
A \emph{split graph} is a 
graph in which the vertices can be partitioned into a clique and an 
independent set.
Recall that a bipartite graph is chordal bipartite if it does not have any  
induced cycle of length greater than 4. The following characterization of
chordal bipartite 
graphs appears in \cite{Brand}.\com{\cite{BLS}.}
\begin{lemma}\label{splitlemma}
Let $G$ be a bipartite graph with bipartition $\{A,B\}$.
\com{Let $G_A$ ($G_B$) be the split graph obtained from $G$ by 
adding edges between every pair of vertices in $A$ ($B$).} Then, $G$ is
chordal bipartite if and only if $C_A(G)$ is strongly chordal.
\com{ if and only if $C_B(G)$ is strongly chordal.}
\end{lemma}
\subsection{Bipartite Powers}
For any two vertices $u,v$ in a graph $G$, let 
$d_G(u,v)$ denote the length of a shortest $u\mbox{-}v$ path in $G$. 
Given a bipartite graph $G$ and an odd positive integer $k$, we define the
graph $G^{[k]}$ to be the graph with $V\left(G^{[k]}\right)=V(G)$ and
$E\left(G^{[k]}\right)=\{(u,v)
~|~ u,v\in V(G),\, d_G(u,v)\mbox{ is odd, and } d_G(u,v)\leq k\}$.
\com{For any odd positive integer 
$k$, we define the $k$-th \emph{bipartite 
power} of a bipartite graph $G$, denoted by $G^{[k]}$, as 
$V(G^{[k]}) = V(G)$ 
and $E(G^{[k]}) = \{(u,v)~|~u,v \in V(G) \mbox{ and } d_G(u,v) \leq 
k \mbox{ is an odd number}\}$.} The graph $G^{[k]}$ is called the
\emph{$k$-th bipartite power} of the bipartite graph $G$.
It is easy to see that if $G$ is a bipartite graph with the bipartition
$\{A,B\}$, then $G^{[k]}$ is also a bipartite graph with the bipartition
$\{A,B\}$.
\section{Bipartite powers of Trees}
Let $T$ be a rooted tree with vertex $r$ being its root. $T$ is 
therefore a bipartite graph and let $\{A,B\}$ be its bipartition. 
For any $u,v \in V(T)$, we say $u \preceq v$ in $T$, if $v \in rTu$.
Otherwise, we say $u \npreceq v$.
For $u,v\in V(T)$, we define $P(u,v):=\{x\in V(T)~|~u\preceq x\mbox{ and
}v\preceq x\}$. The \emph{least common ancestor} (LCA) of any two vertices
$u,v \in V(T)$ in $T$ is that vertex $z \in P(u,v)$ such that $\forall y\in
P(u,v), z\preceq y$. Note that if $z$ is the LCA of $u$ and $v$, then $z \in
uTv, z \in uTr$ and $z \in vTr$.
\com{For any $u,v \in V(T)$,}
We say that a vertex $u$ is \emph{farthest} from a vertex $v$ in $T$ if
$\forall w \in V(T)$, $d_T(v,w) \leq d_T(v,u)$. Note that in this case $u$ will 
be a leaf vertex in $T$. 
\begin{lemma}\label{Powerlemma}
Let $x \in V(T)$ such that $x$ is a leaf vertex in $T$. Then, 
$\left(T- \{x\}\right)^{[k]} = T^{[k]}- \{x\}$. 
\end{lemma}
\begin{proof}
For ease of notation, let $G = \left(T- \{x\}\right)^{[k]}$ and 
$G' = T^{[k]}- \{x\}$. 
Let $(u,v)\in E(G')$. Since $x$ is a 
leaf vertex in $T$, $x \notin uTv$. Therefore, $(u,v) \in E(G)$. Hence, 
$(u,v)\in E(G') \Rightarrow (u,v) \in E(G)$. Also, $(u,v) \in E(G) \Rightarrow
(u,v)\in E(G')$ since $G$ is a subgraph of $G'$. Therefore, 
$(u,v) \in E(G) \Leftrightarrow (u,v)\in E(G')$. This proves the lemma.
\hfill \qed 
\end{proof}

\begin{lemma}\label{Stronglemma}
Let $x\in V(T)$ such that $x$ is farthest from $r$ in $T$. Assume that
$x\in A$. For any odd positive integer $k$, let $G:=C_B(T^{[k]})$.
\com{ be the 
split graph obtained from the bipartite graph $T^{[k]}$ by adding edges between
every pair of vertices in $B$.}
Then, $x$ is a simple vertex in $G$. 
\end{lemma}
\begin{proof}
We shall prove this by proving that, in $G$, for any two vertices
$u_1,u_2 \in N[x]$, such that $d_T(r,u_1) \geq d_T(r,u_2)$, $N[u_1] \subseteq
N[u_2]$. That $x$ is farthest from $r$ in $T$ implies that $u_2\not=x$ (note
that if $d_T(r,u_1) = d_T(r,u_2)$, $u_1$ and $u_2$ are interchangeable).
Now, since $N[x]\cap A=\{x\}$, we have $u_2\in B$.
Let $v \in N[u_1]$ in $G$. If $v \in B$, then $v \in N[u_2]$ (since $B$ induces
a clique in $G$). When $v \notin B$, we split the proof into the 
two cases given below.

\noindent Let $w$ be the LCA of $u_1$ and $u_2$ in $T$.
\begin{case}
$v \preceq w$ in $T$ 
\end{case}
We know that since $u_1,u_2 \in N[x]$ in $G$, $d_T(u_1,x) \leq k$ and 
$d_T(u_2,x) \leq k$. It is easy to see that $w \in u_1Tx$ or 
$w \in u_2Tx$. If $w \in u_1Tx$ then, $d_T(u_1,w) + d_T(w,x) = 
d_T(u_1,x) \leq k$. Otherwise, $d_T(u_2,w) + d_T(w,x) = 
d_T(u_2,x) \leq k$. Since $d_T(r,u_1) \geq d_T(r,u_2)$ implies that
$d_T(u_1,w) \geq d_T(u_2,w)$, we always 
have $d_T(u_2,w) + d_T(w,x) \leq k$. We know that $d_T(r,x) \leq 
d_T(r,w) + d_T(w,x)$ and $d_T(r,v) = d_T(r,w) + d_T(w,v)$. Since 
$d_T(r,v) \leq d_T(r,x)$ (as $x$ is farthest from $r$ in $T$), we have
$d_T(w,v) \leq d_T(w,x)$. Therefore, 
$d_T(u_2,w) + d_T(w,v) \leq k$. Hence, $v \in N[u_2]$ in $G$.  
\begin{case}
$v \npreceq w$ in $T$
\end{case}
In this case, it is easy to see that $w \in u_1Tv$ and $w \in u_2Tv$.
This implies
that $d_T(u_1,v)=d_T(u_1,w)+d_T(w,v)$ and  $d_T(u_2,v)=d_T(u_2,w)+d_T(w,v)$.
Since $d_T(u_1,v) \leq k$ and $d_T(u_1,w) \geq d_T(u_2,w)$, we can conclude 
that $d_T(u_2,v)\leq k$. Therefore, $v \in N[u_2]$ in $G$.\vspace{0.1in}

This proves that $N[u_1] \subseteq N[u_2]$. Hence the lemma.
\hfill\qed
\end{proof}
\begin{theorem}\label{CBGPowertheorem}
For any odd positive integer $k$, $T^{[k]}$ is a CBG.
\end{theorem}
\begin{proof}
Let us prove this by using induction on the number of vertices of 
$T$. Let $x \in V(T)$ such that $x$ is farthest from $r$ in $T$.
Assume $x\in A$. Let $G:=C_B(T^{[k]})$.
\com{be the split graph obtained from the bipartite graph 
$T^{[k]}$ by connecting every pair of vertices in $B$.}
Then by Lemma \ref{Stronglemma}, $x$ is a simple vertex in $G$. Let 
$G' = G - \{x\}$. Note that $G'=C_B(T^{[k]}- \{x\})$.
\com{ is the split graph obtained from 
the bipartite graph $T^{[k]}- \{x\}$ by connecting every pair of vertices 
in $B$.}
But from Lemma \ref{Powerlemma}, $T^{[k]}- \{x\} = 
\left(T- \{x\}\right)^{[k]}$ which by our induction hypothesis 
is a CBG. Then, applying Lemma \ref{splitlemma}, we can say that $G'$ is 
a strongly chordal graph. Since $x$ is a simple vertex in $G$ and since 
$G' = G - \{x\}$, by applying Lemma \ref{stronglemma} we can say that 
$G$ is also a strongly chordal graph. Therefore by Lemma \ref{splitlemma}, 
$T^{[k]}$ is a CBG.\hfill\bbox
\end{proof}

\section{Boxicity of CBGs}
\begin{lemma}\label{Hellylemma}
In an interval graph $I$, let $S \subseteq V(I)$ be a set of 
vertices that induces a clique. Then for an interval representation $f$ of $I$,
$\exists y,z \in \mathbb{R}$ with $y \leq z$ such that $S$ overlaps in the
region $[y,z]$ in $f$.
\end{lemma}
\begin{proof}
Proof of the lemma follows directly from the Helly property for intervals on
the real line.
\hfill\qed
\end{proof}
\begin{lemma}\label{CoBiplemma}
Let $G$ be a bipartite graph with bipartition $\{A,B\}$. 
Let $G'$ be the graph with 
$V(G')=V(G)=V$ and $E(G') = E(G) \cup \{(u,v)~|~u,v\in 
A \mbox{ and } u \neq v\} \cup \{(u,v)~|~u,v\in 
B \mbox{ and } u \neq v\}$. Then, 
$\boxi(G) \geq \frac{\boxi(G')}{2}$.
\end{lemma}
\begin{proof}
Let $\boxi(G)=b$. Then by Lemma \ref{robertslem} we have a 
set of interval graphs, say $\mathcal{I} = \{I_1, I_2, \ldots , I_b\}$ with
$V(I_1)=\cdots =V(I_b)=V$, such that $G = I_1 \cap I_2 \cap \cdots \cap 
I_b$. As each $I_i$ is an interval graph, there exists an interval
representation $f_i$ for it.
\com{Let $f_1,\ldots,f_b$ be functions such
that $f_i$ is an interval representation of the interval graph $I_i$.}
For each $i$, let $s_i = \min_{x \in V}
l(f_i(x))$ and $t_i = \max_{x \in V} r(f_i(x))$. Corresponding to each
interval graph $I_i$ in $\mathcal{I}$, we construct two interval graphs $I_i'$
and $I_i''$. We construct the interval representations $f_i'$ and $f_i''$ for
$I_i'$ and $I_i''$respectively from $f_i$ as follows:\vspace{0.1in}

\noindent Construction of $f_i'$:
\begin{tabbing}
~~~~~~~~~~~\=\kill
\>$\forall u \in A$, $f_i'(u) = [s_i,r(f_i(u))]$.\\
\>$\forall u \in B$, $f_i'(u) = [l(f_i(u)),t_i]$.
\end{tabbing}

\noindent Construction of $f_i''$:
\begin{tabbing}
~~~~~~~~~~~\=\kill
\>$\forall u \in A$, $f_i''(u) = [l(f_i(u)),t_i]$.\\
\>$\forall u \in B$, $f_i''(u) = [s_i,r(f_i(u))]$.
\end{tabbing}

Let $\mathcal{I'} = \{I_1',\,\ldots ,\,I_b',\,I_1'',\,\ldots,\, I_b''\}$. 
Now we claim that $G' = \bigcap_{I\in \mathcal{I'}} I$.

\com{$I_1' \cap \cdots \cap I_b' \cap I_1'' 
\cap \cdots \cap I_b''$.}
Let $i\in \mathbb{N}$ such that $1\leq i\leq b$. Since $A$ overlaps
in the region $[s_i,s_i]$ in
$f_i'$, the vertices in $A$ induce a clique in $I_i'$.
Similarly, $B$ overlaps in the region $[t_i,t_i]$ in $f_i'$ and
hence the vertices in $B$ also induce a clique in $I_i'$.
Also, for any $u \in A,\,v \in B$, $(u,v) \in E(G') \Rightarrow (u,v) \in E(G)
\Rightarrow (u,v) \in E(I_i) \Rightarrow (u,v) \in 
E(I_i')$ (since $\forall u\in V,\,f_i(u)\subseteq f_i'(u)$).
\com{(follows from the way $I_i'$ is constructed from 
$I_i$.)}
Hence each $I_i'$ is a supergraph of $G'$. It can be shown by proceeding
along similar lines that each $I_i''$ is a supergraph of $G'$. Therefore,
each $I \in \mathcal{I'}$ is a supergraph of $G'$.

Let $u,v\in V$ such that $(u,v)\not\in E(G')$. Both $u$ and $v$ together cannot
be in $A$ or $B$ since both $A$ and $B$ induce cliques in $G'$. Assume without
loss of generality that $u\in A$ and $v\in B$.
Now, $(u,v) \notin E(G') \Rightarrow (u,v) \notin E(G)$ (follows from the
way $G'$ is constructed from $G$). Since $(u,v) \notin E(G)$, 
$\exists I_i \in \mathcal{I}$ such that $(u,v) \notin E(I_i)$, 
i.e. $f_i(u) \cap f_i(v) = \emptyset$. If $f_i(u)<f_i(v)$ then
$f_i'(u)< f_i'(v)$ and therefore $(u,v) \notin E(I_i')$. 
Otherwise, if $f_i(v) < f_i(u)$ then $f_i''(v)<f_i''(u)$ and therefore $(u,v)
\notin E(I_i'')$. To summarise, for any $(u,v) 
\notin E(G')$, $\exists I \in \mathcal{I'}$ such that $(u,v) 
\notin E(I)$. 

Hence we prove the claim that $G' = \bigcap_{I\in \mathcal{I'}} I$.
\com{$I_1' \cap\cdots \cap I_b' \cap I_1''\cap \cdots \cap I_b''$.}
By Lemma \ref{robertslem}, this means that $\boxi(G') \leq 2b = 2\cdot
\boxi(G)$.
\hfill\qed
\end{proof}

\begin{figure}[h]
\begin{center}
\input{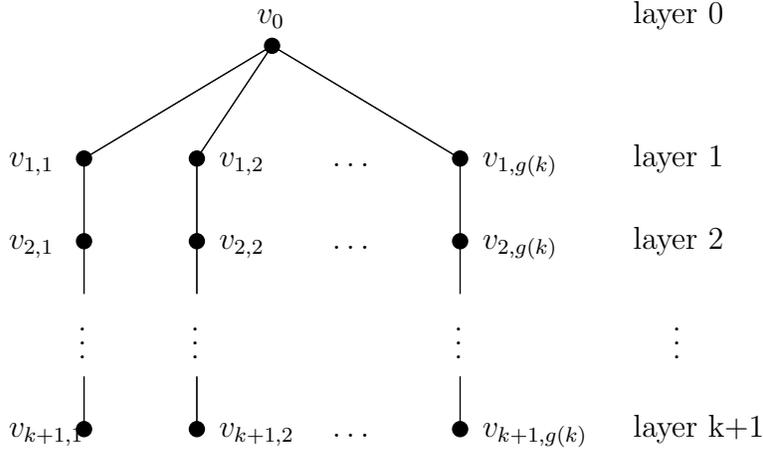}
\caption{Tree $T_k$}
\label{TkFigure}
\end{center}
\end{figure}
\vspace{0.1in}
Let $T_k$ be the tree shown in figure
\ref{TkFigure}. Here $k \in \mathbb{N}$ 
is an odd number and $g(k) = \frac{k+1}{2}\cdot\left(g(k-2) -1\right)+1$ 
with $g(1)=2$. Let 
$G_k = T_k^{[k]}$. It follows from Theorem \ref{CBGPowertheorem} that  
$G_k$ is a CBG. Let $L_i = \{v_{i,1}, v_{i,2}, \ldots , v_{i,g(k)}\}$ denote
the set of all vertices in layer $i$ of $T_k$. Note that $T_k$, and
consequently $G_k$, is a bipartite graph with the bipartition $\{A,B\}$
where $A = \{u \in L_i~|~0 \leq i \leq k+1,\, i \mbox{ is an odd number}\}$ 
and $B = \{u \in L_i~|~0 \leq i \leq k+1,\, i \mbox{ is an even 
number}\}$.
\vspace{0.1in}

Let $G_k'$ be the graph with $V(G_k')=V(G_k)=V(T_k)$ and 
$E(G_k')=E(G_k) \cup$\linebreak $\{(u,w)~|~u,w\in A  
\mbox{ and } u \neq w\} \cup \{(u, w)~|~u,w\in B \mbox{ and } u\neq w\}$. 
So in $G_k'$, $A$ and $B$ induce cliques. 
\begin{lemma}\label{auxlemma}
$\boxi(G_k') > \frac{k+1}{2}$
\end{lemma}
\begin{proof}
Let $X_k':=G_k'-\{v_0\}$. Since $X_k'$ is an induced subgraph of $G_k'$, 
$\boxi(G_k') \geq \boxi(X_k')$. Now, we prove that $\boxi(X_k') > 
\frac{k+1}{2}$ by using induction on $k$. Since $X_1'$ is precisely 
a cycle of length 4, it is not a chordal graph. Therefore, by Lemma 
\ref{lekkerlem}, $\boxi(X_1') >\com{ \frac{1+1}{2} =} 1$ and thus our
induction hypothesis holds for the case $k=1$. Let $m \in 
\mathbb{N}$ be an odd number. Let us assume that our claim is true 
when $k$ is an odd number and $k<m$. Now, when $k=m$, we need 
to prove that $\boxi(X_m') > \frac{m+1}{2}$. We prove this by 
contradiction. Assume $\boxi(X_m') \leq r = \frac{m+1}{2}$. Then 
by Lemma \ref{robertslem}, there exists a set of $r$ interval 
graphs, say $\mathcal{I} = \{I_1, I_2, \ldots , I_r\}$, such 
that $X_m' = I_1 \cap I_2 \cap \cdots \cap I_r$. As each $I_i$ is an interval
graph, there exists an interval representation $f_i$ for it.
\com{Let $f_1,\ldots,f_r$ be
functions such that for each $i$, $f_i$ is an interval representation of $I_i$.}

Since the vertices of $L_1$ induce a clique in $X_m'$, they also 
induce a clique in each $I_i \in \mathcal{I}$. 
Let $[s_i,t_i]=\bigcap_{u\in L_1} f_i(u)$. Lemma \ref{Hellylemma} guarantees
that $[s_i,t_i]\not=\emptyset$.
\com{So by Lemma 
\ref{Hellylemma}, associated with each interval graph $I_i$, 
$\exists s_i, t_i \in \mathbb{R}$ with $s_i \leq t_i$ such that 
$L_1$ overlaps in the region $[s_i,t_i]$ in $f_i$.}
We know 
that for each $v_{m+1,p} \in L_{m+1}$, there exists $v_{1,q} \in L_1$ with 
$p \neq q$ such that $(v_{m+1,p},v_{1,q}) \notin E(X_m')$. So for each
$v_{m+1,p} \in L_{m+1}$, there exists some $I_j \in \mathcal{I}$ 
such that $(v_{m+1,p},v_{1,q}) \notin E(I_j)$. Therefore,
$f_j(v_{m+1,p}) \cap [s_j,t_j] = \emptyset$. 

Now let us partition $L_{m+1}$ into $r$ sets 
$P_1, P_2, \ldots , P_r$ such that $P_i = \{u \in L_{m+1}~|~
f_i(u) \cap [s_i,t_i] = \emptyset \mbox{ and for any 
} j < i,\, f_j(u) \cap [s_j,t_j] \neq \emptyset\}$. Since 
$|L_{m+1}| = g(m)=r\cdot \left( g(m-2)-1 \right) + 1$, there exists 
some $P_j$ such that $|P_j| \geq g(m-2)$. Assume $P_j = P_1$. 
Without loss of generality, let us also assume that $v_{m+1,1},
v_{m+1,2}, $ $\ldots ,v_{m+1,g(m-2)} \in P_1$. So $f_1(v_{m+1,1})
\cap [s_1,t_1] = \emptyset,\, f_1(v_{m+1,2}) \cap 
 [s_1,t_1] = \emptyset,\, \ldots ,$ $f_1(v_{m+1,g(m-2)}) 
\cap  [s_1,t_1] = \emptyset$. Since $X_m' = I_1 \cap I_2 \cap 
\cdots \cap I_r$, both the sets $A$ and $B$ which induce cliques in 
$X_m'$ also induce cliques in $I_1$.
Let $[y_A,z_A]=\bigcap_{u\in A} f_1(u)$ and $[y_B,z_B]=\bigcap_{u\in B}
f_1(u)$. By Lemma \ref{Hellylemma}, $[y_A,z_A]\not=\emptyset$ and
$[y_B,z_B]\not=\emptyset$.
\com{
So by Lemma \ref{Hellylemma}, 
 $\exists y_A, z_A \in \mathbb{R}$ with $y_A \leq z_A$ such that 
$A$ overlaps in the region $[y_A,z_A]$ in $f_1$. Similarily, 
 $\exists y_B, z_B \in \mathbb{R}$ with $y_B \leq z_B$ such that 
$B$ overlaps in the region $[y_B,z_B]$ in $f_1$.}
Since $I_1$ is not a complete graph (recall that $f_1(v_{m+1,1})\cap [s_1,t_1]
=\emptyset$ and therefore $v_{m+1,1}$ is not adjacent to some vertex in $L_1$
in $I_1$), $[y_A,z_A]\cap [y_B,z_B]=\emptyset$ implying that
either $z_A<y_B$ or $z_B<y_A$. Assume $z_A < y_B$ (the proof is similar
even otherwise). Therefore we have, 
\begin{equation}
\label{ineq1}
y_A \leq z_A  <  y_B \leq z_B.
\end{equation}
Since $L_1 \subseteq A$, we have 
\begin{equation}
\label{ineq2}
s_1 \leq y_A  \leq  z_A  \leq t_1.
\end{equation}
For any $1 \leq i \leq m+1$, let $L_i' = \{v_{i,1}, v_{i,2}, \ldots , 
v_{i,g(m-2)}\}$. 
Since $L_{m+1}' \subseteq L_{m+1} \subseteq B$, for any $u \in L_{m+1}'$, 
\begin{equation}
\label{ineq3}
l(f_1(u)) \leq y_B \leq z_B \leq r(f_1(u)).
\end{equation}
From inequalities \ref{ineq1}, \ref{ineq2} and \ref{ineq3}, we can
 see that for any $u \in L_{m+1}'$, 
\begin{equation}
\label{ineq4}
s_1  <  r(f_1(u)).
\end{equation}
Note that $L_{m+1}' \subseteq P_1$ and hence for any 
$u \in L_{m+1}'$, $f_1(u) \cap [s_1,t_1] = \emptyset$, i.e. 
either 
\[t_1 < l(f_1(u)) \mbox{ or } r(f_1(u)) < s_1\] 
is true. From inequality \ref{ineq4}, we then conclude that for any $u \in
L_{m+1}'$,
\begin{equation}
\label{ineq5}
t_1  <  l(f_1(u)).
\end{equation}
Let $l_{min}=\min_{u \in L_{m+1}'} 
l(f_1(u))$. Combining inequalities \ref{ineq2}, \ref{ineq3} and \ref{ineq5}, we
have 
\begin{equation}
\label{mainineq}
s_1 \leq y_A \leq z_A \leq t_1 < l_{min} \leq y_B \leq z_B.
\end{equation}
Let $A' = \{u \in L_i'~|~2\leq i \leq m \mbox{ and } i \mbox{ is 
an odd number}\}$ and  $B' = \{u \in L_i'~|~2\leq i \leq m \mbox{ 
and } i \mbox{ is an even number}\}$. Clearly, $A' \subseteq A$ 
($B' \subseteq B$) and therefore it overlaps in the region 
$[y_A,z_A]$ ($[y_B,z_B]$) in $f_1$. For any $v_{i,j} \in 
A'$, since $(v_{i,j}, v_{m+1,j}) \in E(X_m') \subseteq E(I_1)$, 
\begin{equation}
\label{ineq6}
l(f_1(v_{i,j})) \leq y_A \leq z_A \leq t_1 < l_{min} \leq l(f_1(v_{m+1,j})) \leq
r(f_1(v_{i,j})).
\end{equation}
Also, since for any $u\in B', w \in L_1$, $(u,w) \in E(X_m') \subseteq E(I_1)$,
we have for any $u\in B'$, $f_1(u)\cap [s_1,t_1]\not=\emptyset$ and therefore,
\begin{equation}
\label{ineq7}
l(f_1(u)) \leq t_1 < l_{min} \leq y_B \leq z_B \leq r(f_1(u)).
\end{equation}
From inequalities \ref{ineq6} and \ref{ineq7}, we can say that 
$A' \cup B'$ overlaps in the region $[t_1,t_1]$ in $f_1$. 
Hence in $I_1$, $A' \cup B'$ induces a clique. Now, 
we claim that the graph induced by $A'\cup B'$ in  
$X_m'$, say $Z$, is isomorphic to $X_{m-2}'$. Let $V(X_{m-2}') = 
\{\overline{v}_{1,1}, \ldots ,\overline{v}_{1,g(m-2)}, 
\overline{v}_{2,1}, \ldots ,\overline{v}_{2,g(m-2)}, \ldots,
\overline{v}_{m-1,1}, \ldots ,\overline{v}_{m-1,g(m-2)}\}$. 
Then it can be easily verified that this isomorphism is 
given by the mapping $h:V(Z) \rightarrow V(X_{m-2}')$ where, for any 
$v_{i,j} \in V(Z)$, $h(v_{i,j}) = \overline{v}_{i-1,j}$. Since $X_m'
= I_1 \cap \cdots \cap I_r$, $Z$ is the graph induced by 
$A'\cup B'$ in $I_1 \cap \cdots \cap I_r$. But, 
we have showed that $A'\cup B'$ induces a clique 
in $I_1$ which means that $(u,w)\in E(Z)\Leftrightarrow (u,w)\in
E(I_2\cap\cdots\cap I_r)$. Hence, $Z$ is the graph induced by 
$A'\cup B'$ in $I_2 \cap \cdots \cap I_r$. 
 Therefore, $\frac{m-1}{2}=r-1 \geq \boxi(I_2\cap\cdots\cap I_r)\geq \boxi(Z) 
= \boxi(X_{m-2}')$. But this contradicts our induction 
hypothesis that $\boxi(X_{m-2}') > \frac{m-1}{2}$. Hence 
we prove the lemma.
\hfill \qed 
\end{proof}
From Lemma \ref{CoBiplemma} and Lemma \ref{auxlemma}, we get the 
following lemma. 
\begin{lemma}\label{CBGBoxlemma}
$\boxi(G_k) > \frac{k+1}{4}$.
\end{lemma} 
\begin{theorem}\label{CBGBoxtheorem}
For any $b \in \mathbb{N}^+$, there exists a CBG $G$ with $\boxi(G) > b$. 
\end{theorem}
\begin{proof}
For any odd positive integer $k$, since $G_k$ is the bipartite power of a 
tree $T_k$, $G_k$ is a CBG by Theorem \ref{CBGPowertheorem}.  
Let $G = G_{(4b-1)}$. Then by Lemma \ref{CBGBoxlemma}, $\boxi(G) 
> b$. 
\null\hfill\bbox
\end{proof}


\begin{thebibliography}{10}

\bibitem{Bellantoni}
S.~Bellantoni, I.~Ben-Arroyo Hartman, T.~Przytycka, and S.~Whitesides.
\newblock Grid intersection graphs and boxicity.
\newblock {\em Discrete mathematics}, 114(1-3):41--49, April 1993.

\bibitem{Brand}
Andreas Brandst\"adt.
\newblock Classes of bipartite graphs related to chordal graphs.
\newblock {\em Discrete Applied Mathematics}, 32(1):51--60, June 1991.

\bibitem{BLS}
Andreas Brandst\"adt, V.~Bang Le, and Jeremy~P. Spinrad.
\newblock {\em Graph Classes: {A} Survey}.
\newblock SIAM, Philadelphia, -- edition, 1999.

\bibitem{tech-rep}
L.~Sunil Chandran, Mathew~C. Francis, and Naveen Sivadasan.
\newblock Geometric representation of graphs in low dimension.
\newblock To appear in Algorithmica, available at
  http://arxiv.org/abs/cs.DM/0605013.

\bibitem{CFNMaxdeg}
L.~Sunil Chandran, Mathew~C. Francis, and Naveen Sivadasan.
\newblock Boxicity and maximum degree.
\newblock {\em Journal of Combinatorial Theory, Series B}, 98(2):443--445,
  March 2008.

\bibitem{CN05}
L.~Sunil Chandran and Naveen Sivadasan.
\newblock Boxicity and treewidth.
\newblock {\em Journal of Combinatorial Theory, Series B}, 97(5):733--744,
  September 2007.

\bibitem{Coz}
M.~B. Cozzens.
\newblock Higher and multidimensional analogues of interval graphs.
\newblock Ph. D. thesis, Rutgers University, New Brunswick, NJ, 1981.

\bibitem{CozRob}
M.~B. Cozzens and F.~S. Roberts.
\newblock Computing the boxicity of a graph by covering its complement by
  cointerval graphs.
\newblock {\em Discrete Applied Mathematics}, 6:217--228, 1983.

\bibitem{Farber83}
Martin Farber.
\newblock Characterizations of strongly chordal graphs.
\newblock {\em Discrete Mathematics}, 43(2-3):173--189, 1983.

\bibitem{Golu}
Martin~C Golumbic.
\newblock {\em Algorithmic Graph Theory And Perfect Graphs}.
\newblock Academic Press, New York, 1980.

\bibitem{GoluGoss}
M.C. Golumbic and C.F. Goss.
\newblock Perfect elimination and chordal bipartite graphs.
\newblock {\em J. Graph Theory 2}, pages 155--163, 1978.

\bibitem{OOY}
Yota Otachi, Yoshio Okamoto, and Koichi Yamazaki.
\newblock Relationships between the class of unit grid intersection graphs and
  other classes of bipartite graphs.
\newblock {\em Discrete Applied Mathematics}, 155:2383--2390, October 2007.

\bibitem{Roberts}
F.~S. Roberts.
\newblock {\em Recent Progresses in Combinatorics}, chapter On the boxicity and
  cubicity of a graph, pages 301--310.
\newblock Academic Press, New York, 1969.

\bibitem{Scheiner}
E.~R. Scheinerman.
\newblock Intersection classes and multiple intersection parameters.
\newblock Ph. D. thesis, Princeton University, 1984.

\bibitem{Thoma1}
C.~Thomassen.
\newblock Interval representations of planar graphs.
\newblock {\em Journal of Combinatorial Theory, Series B}, 40:9--20, 1986.

\end{thebibliography}
\end{document}